\documentclass[12pt]{amsart}

\usepackage{graphicx}
\usepackage{amsmath,amssymb,a4}

\theoremstyle{plain}
\newtheorem{theorem}{Theorem}[section]
\newtheorem{lemma}[theorem]{Lemma}
\newtheorem{corollary}[theorem]{Corollary}
\newtheorem{proposition}[theorem]{Proposition}
\theoremstyle{definition}
\newtheorem{remark}[theorem]{Remark}

\newtheorem{definition}[theorem]{Definition}

\numberwithin{equation}{section}

\def\la{\lambda}

\def\R{\mathbb R}

\begin{document}

 \baselineskip=17pt    

\title{Lorentz-Shimogaki and Boyd theorems for weighted Lorentz spaces} 
\author {Elona Agora, Jorge Antezana,  Mar\'{\i}a J. Carro,  and Javier Soria} 

\address{E. Agora, Department of Applied Mathematics and Analysis, University of Barcelona, 
 08007 Barcelona, Spain.}
\email{elona.agora@gmail.com}

\address{J. Antezana, Department of Mathematics, 
Faculty of Exact Sciences, 
National University of La Plata, 
1900 La Plata, Argentina.}
\email{jaantezana@yahoo.com.ar}

\address{M. J. Carro, Department of Applied Mathematics and Analysis, University of Barcelona, 
 08007 Barcelona, Spain.}
\email{carro@ub.edu}

\address{J. Soria, Department of Applied Mathematics and Analysis, University of Barcelona, 
 08007 Barcelona, Spain.}
\email{soria@ub.edu}

\thanks{This work has been partially supported by the Spanish Government Grant MTM2010-14946.} 

\begin{abstract} We prove  the Lorentz-Shimogaki and  Boyd theorems for the spaces $\Lambda^p_u(w)$. 
As a consequence, we give the complete characterization of the strong boundedness of $H$ on these spaces in terms of some geometric conditions on the weights $u$ and $w$, whenever $p>1$.
For these values of $p$, we also give the complete solution of the weak-type boundedness of the Hardy-Littlewood operator on $\Lambda^p_u(w)$. 
\end{abstract}

\date{\today}

\maketitle

\pagestyle{headings}\pagenumbering{arabic}\thispagestyle{plain}

\markboth{Lorentz-Shimogaki and Boyd theorems}{Elona Agora, Jorge Antezana,  Mar\'{\i}a J. Carro, and Javier Soria }

\section{Introduction and motivation}

Given a rearrangement invariant (r.i.) Banach function space  $X$ on $\mathbb R$, the Lorentz-Shimogaki theorem 
(\cite{l3:l3}, \cite{sh:sh} see also \cite [p.\ 154] {bs:bs})   asserts that 
$$
M:X\longrightarrow X \mbox{ is bounded }\qquad\iff\qquad \alpha_X<1, 
$$
where $M$  is the classical  Hardy-Littlewood maximal operator
$$
Mf(x)=\sup_{x\in I}{\frac{1}{|I|}}\int_{I} |f(y)|dy,
$$
(the supremum is taken over all intervals $I$ containing $x\in \R$) and $\alpha_X$ is the upper Boyd index
defined  (\cite{b:b} see also \cite [p.\ 149] {bs:bs})  as follows: 
$$
\alpha_{X}:= \lim_{t \to \infty}\frac{\log {||D_t||_X}}{\log t},
$$
with   
$$
||D_t||_X=\sup_{||f||_X\le 1} ||D_t f||_X, 
$$
  the norm of the dilation operator $D_t f(s)= f(s/t)$. 

Similarly, the classical Boyd theorem  shows  \cite [p.\ 154] {bs:bs}  that 
$$
H:X\longrightarrow X\qquad \mbox{is bounded }\qquad\iff\qquad \alpha_X<1\quad\mbox{and}\quad \beta_X>0, 
$$
where $H$ is the Hilbert transform
$$
Hf(x)=\frac{1}{\pi} \lim_{\varepsilon\to 0^+} \int_{|x-y| > \varepsilon} \frac{f(y)}{x-y}\,dy, 
$$
whenever this limit exists almost everywhere and $\beta_X$ is the lower Boyd index defined by 
$$
\beta_{X}:= \lim_{t \to 0^+}\frac{\log {||D_t||_X}}{\log t}.
$$

In \cite{ms:ms} the Lorentz-Shimogaki and Boyd theorems were extended to the case of  r.i. quasi-Banach spaces.

In a recent paper  \cite{lp:lp}, the upper Boyd index for a general   quasi-Banach function space $X$,     not necessarily r.i.,  was defined using the so-called local maximal operator.  With such definition the classical  Lorentz-Shimogaki theorem was extended to this more general class of spaces.

This  paper is a continuation of the work initiated in  \cite{lp:lp} for a concrete class of quasi-Banach spaces, namely, for weighted Lorentz spaces $\Lambda^p_u(w) $ defined by (see~\cite{l1:l1}, \cite{l2:l2}) 
$$
\Lambda^{p}_{u}(w) =\left\{f\in\mathcal M(\mathbb R): \, ||f||_{\Lambda^{p}_{u}(w)}=\left( \int_0^{\infty}  (f^*_u(t))^p w(t)dt\right) ^ {1/p}< \infty \right\}. 
$$
Here, $\mathcal M(\mathbb R)$ is the class of Lebesgue measurable functions on $\mathbb R$ (we work in dimension one since we shall be concerned with the Hilbert transform), $u$ is a positive and locally integrable function on $\mathbb R$ (we call it weight),  $w$ will also be a weight but defined in $(0, \infty)$,   $f^*_u$  is the decreasing rearrangement of $f$ with respect to the weight $u$  (see \cite{bs:bs}),
$$
f^*_u(t)=\inf\big\{ s>0: u(\{ x\in \mathbb R: |f(x)|>s\})\le t \big\}, 
$$
with $u(E)=\int_E  u(x) dx$ and $0<p<\infty$. We would like to mention   that these spaces include as particular cases the weighted Lebesgue spaces $L^p(u)$ (with $w=1$),  the classical Lorentz spaces $\Lambda^p(w)$ (with $u=1$),  and the  Lorentz spaces $L^{q, p}(u)$ (with $w(t)=t^{p/q -1}$). We shall also need to work with the weak-type space 
$$
\Lambda^{p, \infty}_{u}(w) =\left\{f\in\mathcal M(\mathbb R): \, ||f||_{\Lambda^{p, \infty}_{u}(w)}=\sup_{t>0} f^*_u(t)W^{1/p}(t) < \infty \right\},
$$
where $W(t)=\int_0^t w(s) ds$.

As usual, we shall use the symbol $A\lesssim B$ to indicate that there exists a universal constant $C$, independent of all important parameters, such that $A\le C B$. $A\approx B$ will indicate that 
$A\lesssim B$ and $B\lesssim A$.  If  $E$ is a measurable set and  $u=1$,  we write $u(E)=|E|$. We also recall  that a weight $u$ is in the Muckenhoupt class $A_1$ if 
$Mu(x)\lesssim u(x)$, at almost every point $x\in\mathbb R$. For other definitions (like the  $A_\infty$ class) and further properties about Muckenhoupt  weights we refer to the book~\cite{gr:gr}.

It is known that   the  space $\Lambda^{p}_{u}(w)$  is a  quasi-normed 
space if and only if $w\in \Delta_2$  \cite{cgs:cgs}; that is,
$$
W(2r)\lesssim W(r).
$$
This condition will be assumed all over the paper.  

Concerning the upper Boyd index for these spaces, it was proved in  \cite{lp:lp}  that 
\begin{equation}\label{alfa}
\alpha_{\Lambda^{p}_{u}(w)}=  \lim_{t \to \infty}\frac{\log\overline{W}_u^{1/p}(t)}{\log t}, 
\end{equation}
where, for every $t>1$, 
\begin{equation*}
\overline{W}_{u}(t):=\sup \left\{\frac{W\left(u\Big(\bigcup_{j}I_j\Big)\right)}{W\left(u\Big(\bigcup_{j}S_j\Big)\right)}:
 S_j\subseteq I_j \mbox{ and } |I_j|<t |S_j|,  \mbox{ for every }  j \right\},
\end{equation*}
with  $I_{j}$     disjoint  intervals, $S_j$   measurable subsets,   and all unions are finite. To see (\ref{alfa}), the following result was used:
\begin{theorem} \cite{crs:crs}\label{strongmaximal}
If $0<p<\infty$, 
$$
M:\Lambda^{p}_{u}(w)\longrightarrow \Lambda^{p}_{u}(w)
$$
is bounded if and only if 
 there exists $q\in(0,p)$ such that,  for every finite family of  disjoint intervals $(I_j)_{j=1}^J$,
and every family of measurable sets $(S_j)_{j=1}^{J}$, with $S_j\subset I_j$, for every $j$,  we have that
\begin{equation}\label{raposo}
\frac{W\left(u\left(\bigcup_{j=1}^J I_j\right)\right)}{W\left(u\left(\bigcup_{j=1}^J S_j\right)\right)}
       \lesssim \max_{1\leq j\leq J} \left(\frac{|I_j|}{|S_j|}\right)^q.
\end{equation}
\end{theorem}

\begin{remark} \label{igual}
 For later purposes, it is important to mention that, by regularity and continuity, 
\begin{equation*} 
\overline{W}_{u}(t):=\sup \left\{\frac{W\left(u\Big(\bigcup_{j}I_j\Big)\right)}{W\left(u\Big(\bigcup_{j}S_j\Big)\right)}:
 S_j\subseteq I_j \mbox{ and } |I_j|=t |S_j|,  \mbox{ for every }  j \right\},
\end{equation*}
 where, for every $j$,   $S_j$ is a finite union of intervals. 
\end{remark}

\begin{remark}\label{rem}

\ 

\noindent (i)
We observe that (\ref{raposo}) is equivalent to saying that  there exists $q\in(0,p)$ such that, for every $t>1$, 
$$
\overline{W}_{u}(t) \lesssim t^q. 
$$

\ 

\noindent (ii) It was also proved in \cite{crs:crs} that,  if $0<p<\infty$ and 
\begin{equation}\label{Mweak}
M:\Lambda^{p}_{u}(w)\longrightarrow \Lambda^{p, \infty}_{u}(w)
\end{equation}
is bounded, then $\overline{W}_{u}(t) \lesssim t^p$. Moreover, if $0<p<1$, this condition is sufficient for (\ref{Mweak}), although  this is not the case for other values of $p$. In this paper we shall also give a characterization, in the case $p>1$,    of the weights $u$ and $w$ for which
 $$
 M:{\Lambda^{p}_{u}(w)}\longrightarrow \Lambda^{p, \infty}_{u}(w) 
 $$
 is bounded solving an open problem left in \cite{crs:crs},  (see Theorem~3.2.11). 
\end{remark}

We now describe the main goals of this work: 
\medskip

\noindent
(i)  We give a new proof of the Lorentz-Shimogaki theorem for weighted Lorentz spaces,  without using the local maximal operator (we shall define the upper Boyd index by (\ref{alfa})).

\medskip
\noindent
(ii)   We study whether  the corresponding generalization of the classical Boyd theorem for the Hilbert transform:
$$
H:{\Lambda^{p}_{u}(w)}\longrightarrow \Lambda^{p}_{u}(w)\, \mbox{ is bounded }\quad \iff\quad \beta_{\Lambda^{p}_{u}(w)}>0 \mbox{ and } \alpha_{\Lambda^{p}_{u}(w)}<1
$$
 holds true, where the generalized lower Boyd  index $\beta_{\Lambda^{p}_{u}(w)}$ will be defined later on.
 
 Concerning (ii), we shall prove that this is the case if $p>1$ and, as a consequence, we shall give the complete characterization of the boundedness 
 $$
 H:{\Lambda^{p}_{u}(w)}\longrightarrow \Lambda^{p}_{u}(w) 
 $$
 in the  case $p>1$ in terms of geometric conditions on the weights $u$ and $w$.  
 
 Finally, we shall show that, for every $p>0$, 
$$
\beta_{\Lambda^{p}_{u}(w)}>0 \mbox{ and } \alpha_{\Lambda^{p}_{u}(w)}<1 \implies   H:{\Lambda^{p}_{u}(w)}\longrightarrow \Lambda^{p}_{u}(w)\implies \beta_{\Lambda^{p}_{u}(w)}>0. 
$$

\section{The case $u=1$: $\Lambda^p(w)$}

Let us start by analyzing the case $u=1$ since it was the starting point for our results.  Note that the space $\Lambda^p(w)$ is, in fact, a rearrangement invariant function space. In particular, a simple computation of $||D_t||_{\Lambda^{p}(w)}$ gives us the following result. 

\begin{proposition}  \cite{b:b, ms:ms} For every $0<p<\infty$, 
$$
\alpha_{\Lambda^{p}(w)}:= \lim_{t \to \infty}\frac{\log\overline{W}^{1/p}(t)}{\log t},
$$
and
$$
\beta_{\Lambda^{p}(w)}:= \lim_{t\to 0^+}\frac{\log\overline{W}^{1/p}(t)}{\log t},
$$
where 
$$
\overline{W}(t):=\sup_{s\in(0,+\infty)} \frac{W(st)}{W(s)}.
$$
\end{proposition}

\ 

Then, the Lorentz-Shimogaki theorem \cite{ms:ms} applied to $\Lambda^{p}(w)$ says that 
$$
M:\Lambda^{p}(w)\longrightarrow \Lambda^{p}(w)\quad \mbox{is bounded }\quad\iff\quad \lim_{t \to \infty}\frac{\log\overline{W}^{1/p}(t)}{\log t}<1.
$$

On the other hand, we recall the following result of Ari\~no and Muckenhoupt~\cite{am:am}:

\begin{theorem} For every $0<p<\infty$, 
$$
M:\Lambda^{p}(w)\longrightarrow \Lambda^{p}(w)\quad \mbox{is bounded }\quad\iff\quad  w\in B_p;
$$
that is, for every $r>0$, 
$$
r^p \int_r^{\infty}\frac{w(t)}{t^p} \, dt  \lesssim  \int_0^rw(s) ds. 
$$
\end{theorem}

Consequently, we have the following corollary:

\begin{corollary}\label{ari} For every $0<p<\infty$, 
$$
\lim_{t \to \infty}\frac{\log\overline{W}^{1/p}(t)}{\log t}<1 \,\mbox{ if and only if }\,  w\in B_p. 
$$
\end{corollary}

We shall give a direct proof of this result using the following lemma about submultiplicative functions. 
Observe that  $\overline{W}$ is submultiplicative;  that is, for every $t, s>0$, 
$$
\overline{W}(ts)  \leq\overline{W}(t)\overline{W}(s). 
$$

\begin{lemma}\label{lll} 
For every submultiplicative increasing function $\varphi$ defined in $[1, \infty)$, 
$$
\lim_{t \to \infty}\frac{\log\varphi(t)}{\log t}<1, 
$$
if and only if there exists $\gamma<1$ such that $\varphi(x)\lesssim x^\gamma$, for every $x>1$. 

\end{lemma}

\begin{proof} By hypothesis, there exists $t_0>1$ such that $\varphi(t_0) < t_0$. Now, given $x>1$, 
there exists $k\in\mathbb N$ such that $x\in ( t_0^k, t_0^{k+1})$ and hence, since  $\varphi$ is increasing and submultiplicative, 
$$
\varphi(x)\le\varphi(t_0^{k+1})\le \varphi(t_0)^{k+1}\le t_0 \bigg(\frac{\varphi(t_0)}{t_0}\bigg)^{k+1}  x.
$$
Using that  $c=\frac{\varphi(t_0)}{t_0}<1$, we have that 
$c^ {k+1}\le c^{\frac{\log x}{\log t_0} }= x^{{\frac{\log c}{\log t_0} }}$ with $\log c<0$. Hence, 
$$
\varphi(x)\le t_0 x^{1+{\frac{\log c}{\log t_0} } }\approx x^{\gamma},
$$
with $\gamma<1$. 
Conversely,  if $\varphi(t)\le C t^\gamma$, for every $t>1$, 
$$
\lim_{t \to \infty}\frac{\log\varphi(t)}{\log t}\le \lim_{t \to \infty}\frac{\log (C t^{\gamma})}{\log t} = \gamma<1.
$$
\end{proof}

\begin{proof} [Proof of Corollary~\ref{ari}]
It is enough to apply Lemma~\ref{lll} to the function $\overline W^{1/p}$ and recall that \cite{am:am}:
$$
w\in B_p\qquad\iff\qquad \overline W(t)\lesssim t^q, \, \mbox{ for some } q<p \, \mbox{ and every }  t>1. 
$$
\end{proof}

\ 

Similarly, the Boyd theorem applied to $\Lambda^{p}(w)$ says that 

\ 
\begin{eqnarray*} 
& &H:\Lambda^{p}(w)\longrightarrow \Lambda^{p}(w)  \mbox{ is bounded } \iff 
\\
& &\qquad\qquad\qquad\qquad \lim_{t \to \infty}\frac{\log\overline{W}^{1/p}(t)}{\log t}<1 
\qquad\mbox{ and } \qquad\lim_{t \to 0^+}\frac{\log\overline{W}^{1/p}(t)}{\log t}>0.
\end{eqnarray*}

\ 

On the other hand, we now have the following result \cite{s:s, n:n2}:

\begin{theorem}\label{sssa} For every $0<p<\infty$, 
$$
H:\Lambda^{p}(w)\longrightarrow \Lambda^{p}(w)\quad \mbox{is bounded }\quad\iff\quad  w\in B_p\cap B^*_\infty, 
$$
where the   $B^*_\infty$ class is defined by the following condition: for every $r>0$, 
$$
\int_0^r \frac 1t \int_0^t w(s) ds dt  \lesssim  \int_0^rw(s) ds. 
$$
\end{theorem}

In order to describe the conditions of Theorem~\ref{sssa} in terms of $\overline{W}$, and in view of Corollary~\ref{ari}, it suffices to prove the following result.

\begin{proposition}\label{beta2}
$w\in B^*_\infty$ if and only if
\begin{equation}\label{beta} 
\lim_{t \to 0^+}\frac{\log\overline{W}^{1/p}(t)}{\log t}>0.
\end{equation}

\end{proposition}

The proof of this result is based on the following lemma:

\begin{lemma}\label{clickBstar}
If $\varphi:(0,1]\to[0,1]$ is an increasing submultiplicative function, then the following statements are equivalent:
\begin{enumerate}
\item[(i)] $\displaystyle \varphi(\lambda)<1$ for some $\lambda\in(0,1)$,
\item[(ii)] $\displaystyle \varphi(x)\lesssim \frac{1}{1+\log(1/x)}$,
\item[(iii)] $\displaystyle \lim_{t \to 0^+}\frac{\log\varphi (t)}{\log t}>0$.
\end{enumerate}
\end{lemma}

\begin{proof} Clearly $(ii)$ implies $(i)$ and $(iii)$ implies $(i)$ as well. 

\noindent $(i)\Rightarrow (ii)$ Since $0<\lambda<1$, given $x\in (0, 1)$, there exists $k\in\mathbb N$ such that $x\in (\lambda^{k+1}, \lambda^k)$ and hence, since $\varphi(\lambda)<1$, we have that 
$$
A=\sup_{k\in\mathbb N}  \varphi(\lambda)^k \big(1+(k+1)\log(1/\lambda)\big)<\infty.
$$
Therefore, 
$$
\varphi(x)\le \varphi(\lambda^k)\le \varphi(\lambda)^k\le  \frac{A}{1+(k+1)\log(1/\lambda)}\lesssim  \frac{1}{1+\log(1/x)},
$$
as we wanted to see. 

\noindent
\noindent $(i)\Rightarrow (iii)$ If $x\in (\lambda^{k+1}, \lambda^k)$, we have that $\log \varphi(x) \le k \log \varphi(\lambda)$, and since $(k+1)\log\lambda\le\log x$, we get that
$$
\frac{\log\varphi (x)}{\log x}\ge \frac k{k+1}\frac{\log\varphi (\lambda)}{\log \lambda}\ge \frac{\log\varphi (\lambda)}{2\log \lambda}, 
$$
from which   the result follows. 
\end{proof}

\begin{proof}   [Proof of Proposition~\ref{beta2}]
If $w\in B^*_\infty$,  for every $s\le r$, 
$$
W(s) \log \frac rs\le \int_s^r \frac{W(t)} t dt\lesssim W(r), 
$$
and since $W$ is increasing, we deduce that  $W(s) (1+ \log \frac rs)\lesssim W(r)$, which implies that 
$$
\overline W (y) \lesssim \frac 1{1+\log \frac 1y},
$$ 
for every $0<y\le 1$. Thus, $\overline {W}^{1/p}$
satisfies the hypothesis of Lemma~\ref{clickBstar} and (\ref{beta}) follows.

 Conversely,  if (\ref{beta}) holds, and we write $c=\lim_{t \to 0}\frac{\log\overline{W}^{1/p}(t)}{\log t}$, it is easy  to see that, for $t$ small enough, 
 $\overline {W^{1/p}}(t)\le t^{c/2}$,  and thus there exists $\lambda<1$ satisfying that $\overline {W^{1/2}}(\lambda) <1$. Hence, by Lemma~ \ref{clickBstar}, 
$$
\int_0^r \frac{W(t)}t dt\lesssim W(r) \int_0^r   \Big(1+\log(r/t)\Big)^{-2} \frac{dt}t \lesssim W(r),
$$
and therefore, $w\in B^*_\infty$. 
\end{proof}

\begin{remark}Concerning the function $\overline W_u$, we observe that, if $u=1$ then, for every $t>1$, 
$$
\overline W(t)= \overline W_u(t).
$$
Indeed, it is enough to note that, given any finite family of disjoint intervals $(I_j)_{j=1}^J$  and measurable sets  $(S_j)_{j=1}^{J}$, such that $S_j\subset I_j$ and $|I_j|=t |S_j|$, for every $j$, it holds that 
$$
W\bigg (\Big|\bigcup_{j}S_j \Big|\bigg)= W\bigg (t \Big|\bigcup_{j}I_j \Big|\bigg). 
$$
Since $|\bigcup_{j}I_j|$ can be any positive real number,  by Remark \ref{igual} and the definition of $\overline{W}(t)$, it follows that $\overline W_u(t)$ and $\overline{W}(t)$ have to coincide.
\end{remark}

\section{The Lorentz-Shimogaki theorem for  $\Lambda^p_u(w)$}

As mentioned in the introduction, it was proved in \cite{lp:lp} that 
\begin{equation}\label{bla}
\alpha_{\Lambda^{p}_{u}(w)}= \lim_{t \to \infty}\frac{\log\overline{W}_u^{1/p} (t)}{\log t}.
\end{equation}
To justify the existence of the limit, the authors show that $ \overline W_u$ is pointwise equivalent to a submultiplicative function involving the local maximal function. In the following proposition, we will prove that the function $\overline W_u$ is in fact submultiplicative, which gives a direct proof of this result. With this aim, we need the following technical lemma.

\begin{lemma} \label{covering}
Let $I$ be an interval and let $S=\cup_{k=1}^{N}(a_k,b_k)$ be the union of disjoint intervals such that $S\subset I$.
Then, for every $t\in  [\,1,{|I|}/{|S|} \,]$ there exists a collection of disjoint subintervals $\{I_n\}_{n=1}^M$
satisfying that $S\subset \cup_{n} I_n$ and  such that, for every $n$, 
\begin{equation}\label{coveringform}
t|S\cap I_n|=|I_n|.
\end{equation}
\end{lemma}

\begin{proof}
Without loss of generality we can assume that $I=(0, |I|)$ and also that $a_1<a_2<\dots<a_N$.
First observe that if $J=\cup I_n$ we should in particular obtain  $t|S|=|J|$ applying \eqref{coveringform}.
We use induction in $N$. Clearly it is true for  $N=1$. Indeed, it suffices to consider $0\leq c\leq a_1<b_1\leq d\leq |I|$
such that $t(b_1-a_1)=d-c$. Suppose that the results holds for all $k<N$. We will prove that it also holds for $k=N$.
 
 \noindent
{Case I:} If  $|I|-t|S|\leq a_1$, then  it suffices to consider $I_1=(|I|-t|S|, |I|)=J$. 
\noindent
{Case II:} If $a_1 < |I|-t|S|$, set  $\bar{I}=(a_1, |I|)$. Observe that $t|S|<|\bar{I}|$ and $S\subset \bar{I}$.
Hence in this case we could assume, without loss of generality, that $a_1=0$.  Let now $I_1=(0,c)$ such that $b_1\leq c\leq |I|$
and $t |S\cap I_1|=c=|I_1|$.  Note that $c\notin S$. In fact, suppose that there exists $S_m=(a_m, b_m)$ such that $c\in S_m$.
Then,
$t|S\cap [0,a_m)|>|[0,a_m)|$  which implies that $t|S\cap [0,c)|>|[0,c)|=|I_1|$, which is a contradiction. Therefore, we obtain that 
$$
t|S\cap I_1|+t|S\cap [c,|I|)|=t|S|<|I|=|I_1|+|[c,|I|)|=t|S\cap I_1|+|[c,|I|)|,
$$
and consequently $t|S\cap [c,|I|)|<|[c,|I|)|$.  Then, since $[c,|I|)$ is the union of at most $N-1$  intervals $\{(a_k,b_k)\}_k$, 
we apply the inductive hypothesis to the intervals $[c,|I|)$  and the set $S\cap [c,|I|)$ and  we obtain the intervals
$I_2, \dots, I_M$ such that \eqref{coveringform} is satisfied.

\end{proof}

\begin{lemma}\label{subsub}
The function $\overline{W}_u$ is submultiplicative on $[1, \infty)$. 
\end{lemma}

 \begin{proof} 
 Consider a finite family of intervals $I_j$, and measurable sets $S_j\subseteq I_j$ which are finite union of intervals
such that $|I_j|=\la\mu|S_j|$. 
Then, we can apply Lemma~\ref{covering}
and for each $j$ obtain a set $J_j$ such that it is a union of a finite number of pairwise disjoint  intervals, that we call $J_{ji}$:
$$
S_j\subseteq J_j,\quad \la|S_j\cap J_{ji}|=|J_{ji}|,
          \quad J_j\subseteq {I}_j,\qquad\mbox{ and } \mu |J_j|=|{I}_j|.
$$
So, we have that
\begin{align*}
\frac{W\left(u\Big(\bigcup_{j}I_j\Big)\right)}{W\left(u\Big(\bigcup_{j}S_j\Big)\right)}&=
\frac{W\left(u\Big(\bigcup_{j}I_j\Big)\right)}{ W\left(u\Big(\bigcup_{j} S_j\Big)\right)}
\frac{W\left(u\Big(\bigcup_{j} J_j\Big)\right)}{W\left(u\Big(\bigcup_{j}J_j\Big)\right)}
\leq \overline{W}_u(\la)\overline{W}_u(\mu).
\end{align*}
Therefore, taking supremum over all  possible choices of intervals $I_j$ and measurable subsets $S_j$ such that
$S_j\subseteq I_j$ and $|I_j|= \la\mu|S_j|$, we get that
$\overline{W}_u(\la\mu)\leq \overline{W}_u(\la)\overline{W}_u(\mu)$. \end{proof} 

\begin{definition} Let us define the upper Boyd index for the space $\Lambda^{p}_{u}(w)$ as 
$$
\alpha_{\Lambda^{p}_{u}(w)}=  \lim_{t \to \infty}\frac{\log\overline{W}_u^{1/p}(t)}{\log t}. 
$$
\end{definition}

\begin{remark}
In \cite{lp:lp} the generalized upper Boyd index was introduced in terms of the local maximal operator. Recall that, the local maximal operator $m_\la$ of a measurable function $f$ is defined by
$$
m_\la f(x)=\sup_{x\in I} (f\chi_I)^*(\la|I|),
$$
where $\la\in(0,1)$. In terms of this operator, the upper Boyd index was defined by:
$$
\lim_{\la\to 0^+}\frac{\log \|m_\la\|_{\Lambda^{p}_{u}(w)}}{\log 1/\la}. 
$$
In the original definition for r.i. spaces, the function $t\mapsto \|D_t\|$ is submultiplicative, which justifies the existence of the limit. In the case of the   local maximal operator, it is not known weather or not the function $\la\mapsto \|m_\la\|$ is submultiplicative. In the case of weighted Lorentz spaces, it can be proved that (see \cite[ Lemma 5.1]{lp:lp} and the proof of Theorem 3.2.4 in \cite{crs:crs})
$$
\overline{W}_u(1/\la)\leq \|m_\la\|_{\Lambda^{p}_{u}(w)}^{\, p}\leq \, \overline{W}_u(2/\la).
$$ 
\end{remark}

We can now prove the following extension of the Lorentz-Shimogaki theorem: 

\begin{theorem}\label{popo} \cite{lp:lp}
If $0<p<\infty$, then
$$
M:\Lambda^{p}_{u}(w)\longrightarrow \Lambda^{p}_{u}(w)
$$
is bounded if and only if $\alpha_{\Lambda^{p}_{u}(w)}<1$. 
\end{theorem}

\begin{proof} By Theorem~\ref{strongmaximal} and Remark \ref{rem}, the boundedness of $M$ is equivalent to $\overline W_u(t)\lesssim t^{q}$ for some $q<p$, and since  $\overline W_u$ is increasing and submultiplicative, the result follows from Lemma~\ref{lll}. 
\end{proof}

\section{The Boyd  theorem for  $\Lambda^p_u(w)$}

  In what follows 
$$
H:\Lambda^{p}_u(w)\longrightarrow \Lambda^{p}_u(w)
$$
will indicate that, for every $f\in \Lambda^{p}_u(w)$, $Hf(x)$ is well defined at almost every point $x\in\mathbb R$,  and 
$$
||Hf||_{\Lambda^{p}_u(w)}\lesssim ||f||_{\Lambda^{p}_u(w)},
$$
and similarly for $H:\Lambda^{p}_u(w)\longrightarrow \Lambda^{p, \infty}_u(w)$. 
 
Our next goal is to introduce the definition of the lower Boyd index for $\Lambda^{p}_{u}(w)$ and prove the corresponding Boyd theorem. 

\begin{definition} If $t\in (0,1]$, we define 
\begin{equation*}
\underline{{W}_{u}}(t):=\sup \left\{\frac{W\left(u\Big(\bigcup_{j}S_j\Big)\right)}{W\left(u\Big(\bigcup_{j}I_j\Big)\right)}:\,
            S_j\subseteq I_j  \mbox { and }  |S_j|<t|I_j|, \mbox{  for every } j \right\},
\end{equation*}
where $I_{j}$ are disjoint  intervals and all unions are finite.
\end{definition}

As in Remark \ref{igual} we can substitute $|S_j|<t|I_j|$ by an equality $|S_j|=t|I_j|$ and assume that $S_j$ is the finite union of  intervals.  

With this definition, we can prove  the following result (similar to Lemma~\ref{subsub}). 

\begin{proposition}
The function $\underline{{W}_{u}}$ is submultiplicative in $[0,1]$. 
\end{proposition}

\begin{lemma}\label{mayor1} If $u\in A_\infty$, there exist $C_u>0$ and $\alpha>0$ such that, for every $0<t<1$, 
$$
\overline W(t)\le  \underline{{W}_{u}}(C_ut^\alpha).
$$
\end{lemma}

\begin{proof} It is known that  if $u\in A_\infty$,  there exist $C_u>0$ and $\alpha>0$ such that, for every interval $I$ and every measurable set $E\subset I$, 
$$
\frac{|E|}{|I|}\le C_u \bigg(\frac{u(E)}{u(I)}\bigg)^{\alpha}. 
$$
Now, let $0<t<1$ and let $s>0$. Let $I$ be such that $u(I)=s$ and set $E\subset I$ such that $u(E)=ts$. Then,
$$
\frac{W(ts)}{W(s)}=\frac{W(u(E))}{W(u(I))}\le \underline{{W}_{u}}\bigg(\frac{|E|}{|I|}\bigg)\le 
\underline{{W}_{u}}(C_u t^\alpha),
$$
and the result follows taking the supremum in $s>0$. 
\end{proof}
  
 By analogy with the case of the upper index, we give the following definition (which agrees in the case $u=1$ with the classical one). 
 
 \begin{definition} 
We define the generalized lower Boyd index associated to $\Lambda^p_u(w)$ as
$$
\beta_{\Lambda^p_u(w)}:= \lim_{t\to 0^+}\frac{\log\underline{{W}_{u}}^{1/p}(t)}{\log t}.
$$
\end{definition}

\begin{theorem} \label{boyd1} Let $0<p<\infty$. If 
$$
H:\Lambda^{p}_u(w)\longrightarrow \Lambda^{p}_u(w)
$$
is bounded, then $
 \beta_{\Lambda^p_u(w)}>0.$
\end{theorem}

\begin{proof}  It was proved in \cite{acs2:acs2} that the boundedness of $H$ on $\Lambda^{p}_u(w)$ implies  that $u\in A_\infty$ and  $w\in B^*_\infty$, and hence the result now follows from Proposition~\ref{beta2} and Lemma~\ref{mayor1}. 
\end{proof}

\begin{theorem}\label{boyd2} Let $0<p<\infty$.  If
$$
\alpha_{\Lambda^{p}_{u}(w)}<1\quad\mbox{and }\quad  \beta_{\Lambda^p_u(w)}>0,
$$
then
$$
H:\Lambda^{p}_u(w)\longrightarrow \Lambda^{p}_u(w)
$$
is bounded.
\end{theorem}

\begin{proof} If $ \beta_{\Lambda^p_u(w)}>0$, we have that necessarily $\underline{{W}_{u}}(0+)=0$ and hence, for every $\varepsilon>0$ there exists $\delta>0$ such that $\underline{{W}_{u}}(t)<\varepsilon$, for every $t<\delta$. Consequently, 
for every $\varepsilon>0$ there exists $\delta>0$ such that $W\left(u(S)\right)\leq \varepsilon W(u(I))$, provided $S \subseteq I $ and $|S |\leq \delta |I |$. But this condition was proved in \cite{acs2:acs2} to be equivalent to  $u\in A_\infty$ and $w\in B^*_\infty$. Now,  if $u\in A_\infty$ \cite {bk:bk}, 
$$
(H^*f)^*_u(t)\lesssim \big(Q\,(Mf)^*_u\big) (t/4),
$$
whenever the right hand side is finite, $H^*$ is the Hilbert maximal operator
$$
H^*f(x)= \sup_{\varepsilon>0} \bigg| \int_{|x-y| > \varepsilon} \frac{f(y)}{x-y}\,dy\bigg|, 
$$
and
$$
Qf(t)=\int_t^\infty f(s) \frac{ds}s
$$
is the conjugate Hardy operator. 

Then   (see   \cite{acs2:acs2} for the details), using the facts that, under the condition $w\in B^*_\infty$ we have that $Q$ is bounded on the cone of decreasing functions on $L^{p}(w)$,  and   $M$  is bounded on $\Lambda^p_u(w)$ since $\alpha_{\Lambda^{p}_{u}(w)}<1$, we obtain that $H^*$ is bounded on  $\Lambda^p_u(w)$. Hence, standard techniques show that, for every $f\in\Lambda^p_u(w)$, there exists $Hf(x)$ at almost every $x\in\mathbb R$ and, by Fatou's lemma, we obtain the result. 
\end{proof}

\medskip

Let us now see that, if $p>1$, then we have the converse of the previous result, and so, the Boyd theorem in the context of weighted Lorentz spaces.

\begin{theorem}\label{Boyd en Lorentz}  If  $p> 1$,  then
$$
H:\Lambda^{p}_u(w)\longrightarrow \Lambda^{p}_u(w)
$$
is bounded if and only if 
$$
\alpha_{\Lambda^{p}_{u}(w)}<1\qquad\mbox{and }\qquad  \beta_{\Lambda^p_u(w)}>0. 
$$
\end{theorem}

\medskip

Note that, it only remains to prove that $\alpha_{\Lambda^{p}_{u}(w)}<1$. By Theorem \ref{popo}, it is equivalent to prove
$$
H:\Lambda^{p}_u(w)\longrightarrow \Lambda^{p}_u(w)\implies M:\Lambda^{p}_u(w)\longrightarrow \Lambda^{p}_u(w). 
$$
In \cite{acs2:acs2}, it was proved that 
$$
H:\Lambda^{p}_u(w)\longrightarrow \Lambda^{p, \infty}_u(w)\implies M:\Lambda^{p}_u(w)\longrightarrow \Lambda^{p, \infty}_u(w). 
$$
Since the strong boundedness implies the weak boundedness, Theorem \ref{Boyd en Lorentz} will be proved if, for $p>1$, 
$$
M:\Lambda^{p}_u(w)\longrightarrow \Lambda^{p, \infty}_u(w)\implies M:\Lambda^{p}_u(w)\longrightarrow \Lambda^{p}_u(w). 
$$
This  is a problem of independent interest and it was left open in \cite{crs:crs}. We dedicate the next section to the proof of this result, which will conclude the proof of Theorem \ref{Boyd en Lorentz}.

\section{Weak and strong boundedness of Hardy-Littlewood maximal operator in weighted Lorentz spaces}

The main result of this section is the following theorem:

\begin{theorem}\label{BpiBp}
 If    $p>1$,  then
$$
M:\Lambda^{p}_u(w)\longrightarrow \Lambda^{p, \infty}_u(w)
$$
is bounded if and only if, for some $q<p$,  $\overline W_u (t)\lesssim t^q$  for every $t>1$. 
Consequently, 
$$
M:\Lambda^{p}_u(w)\rightarrow \Lambda^{p, \infty}_u(w)\, \mbox { is bounded}\quad\iff M:\Lambda^{p}_u(w)\rightarrow \Lambda^{p}_u(w)\, \mbox { is bounded}. 
$$
\end{theorem}

\medskip

The idea of the proof is inspired by the article of Neugebauer \cite{n:n}.
Firstly, we need the following lemma: 

\begin{lemma} \label{mod} Given an interval $I$ and a set $S=\cup_{j=1}^m {S_j}$, with $S_j$ pairwise disjoint intervals, there exists a positive function $f_{S,I}$ supported in $I$ satisfying the following conditions: 
\medskip

\noindent
(i) $f_{S,I}(x)=1$, for every $x\in S$. 

\noindent
(ii) $f_{S,I}(x)\ge\frac{|S|}{|I|}$, for every $x\in I$. 

\noindent
(iii) For every $\frac{|S|}{|I|}<\lambda\le 1$, the level set
$$
\{x: f_{S,I}(x)\ge\lambda\}=\cup_k J_{k, \lambda}, 
$$
where  $\{J_{k, \lambda}\}_k$ are pairwise disjoint intervals satisfying 
$$
|S\cap J_{k, \lambda}|=\lambda |J_{k, \lambda}|, 
$$
and there exists $ L_{k,  \lambda}\subset\{1, \cdots, m\}$ such that 
$$
J_{k, \lambda}\cap S= \cup_{l\in L_{k, \lambda} }{S_{j_l}}.
$$
\end{lemma}

\begin{proof}  For simplicity we shall use the following notation: if we have a collection of sets $\{F_j\}_{j=1}^N$, we write $\cup^* F_j$ to indicate the union of a subcollection, whenever  it is not important which subcollection is. Similarly, we write $\sum^* |F_j|$ to indicate that we are summing the measures of the sets of a certain subcollection. We emphasize that the symbols $\cup^*$  or $\sum^*$ in two different places may refer to two different subcollections. 

The proof is done by induction in $m$.  The case $m=1$ is easy since, in this case, if $I=(a, d)$ and $S=(b, c)$
with $a< b<c < d$,  we take, for every $\frac{|S|}{|I|}<\lambda<1$,  $x_\lambda\in (a, b)$ and $y_\lambda\in (c, d)$ such that
$$
\frac{b-x_\lambda}{b-a}= \frac{y_\lambda-c}{d-c}\qquad\mbox{and}\qquad y_\lambda-x_\lambda = \frac 1\lambda (c-b). 
$$
Then, if we define  
$$
J_{1, \lambda}=[x_\lambda, y_\lambda], \qquad\text{if }\, \frac{|S|}{|I|}<\lambda\le1,  
$$
$$
J_{1, \lambda}=\overline I, \qquad \text{if }\, \lambda\le \frac{|S|}{|I|},  
$$
and $J_{1, \lambda}=\emptyset$, if $\lambda>1$,  one can immediately see that if $\lambda_1\le \lambda_2$,  $J_{1, \lambda_2}\subset J_{1, \lambda_1}$, and 
 $$
 J_{1, \lambda}=\bigcap_{\mu<\lambda }J_{1, \mu}.
 $$

Hence, if we define
$$
f_{S,I}(x)=\sup\{\lambda>0: x\in J_{1, \lambda}\},
$$
we obtain that   $\{x: f_{S,I}(x)\ge \lambda\}=J_{1, \lambda}$ and the rest of the properties are easy to verify. The cases where $a=b$ or $c=d$ are done similarly (see Figure~\ref{fig1}).

\begin{figure}[b] 
   \centering
   \includegraphics[width=.75\textwidth]{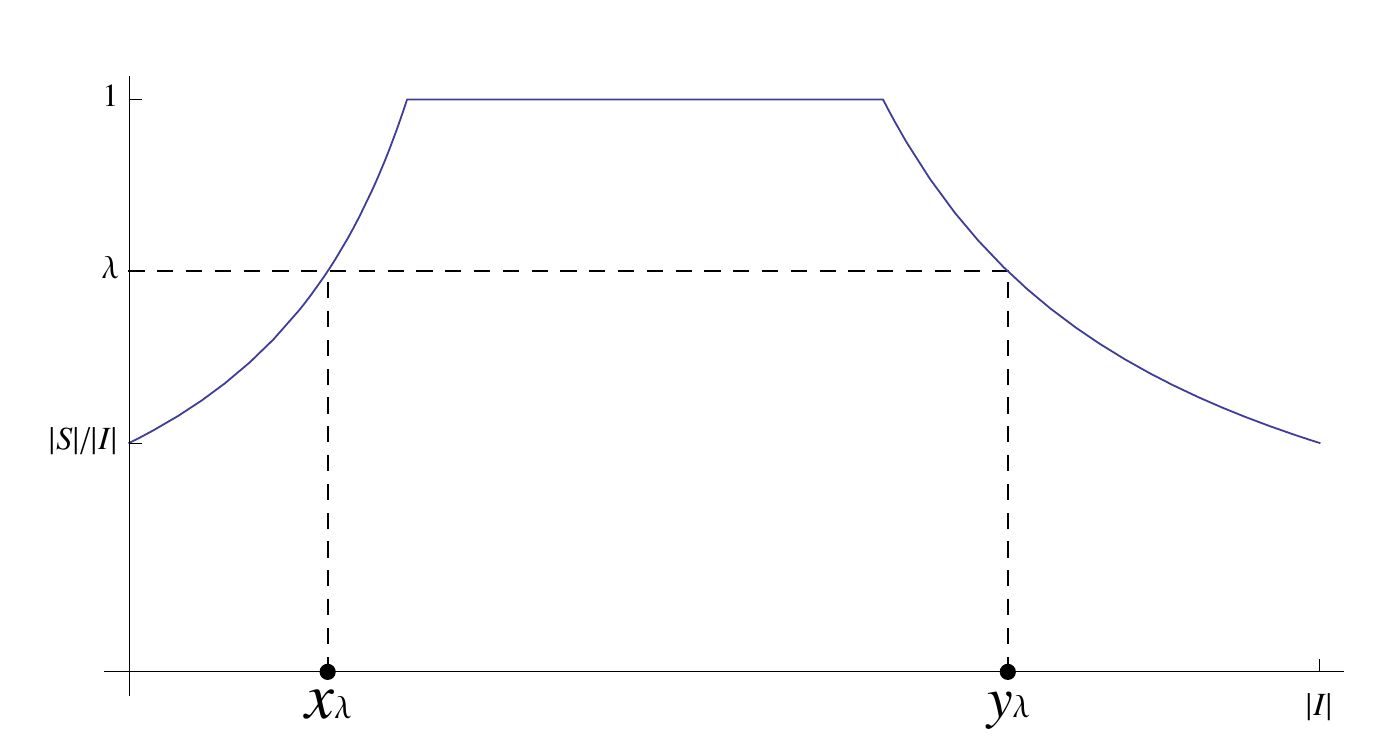} 
   \caption{$f_{S,I}$ when $m=1$.}
   \label{fig1}
\end{figure}

 Let us now assume that the result is true for $m=n$ and let us prove it for $m=n+1$. Let then $S= \cup_{j=1}^{n+1} {S_j}$, with $S_j$ pairwise disjoint intervals.  Let us define $f_k=  f_{S_k, I}$ and 
 $ 
 f_0(x)=\max \Big( \sup_k f_{k}(x), \frac{|S|}{|I|}\Big), 
 $ 
 and let
 $$
 \lambda_0=\inf\bigg\{\frac{|S|}{|I|}\le r<1: \{ f_j(x)\ge r\} \cap \{ f_k(x)\ge r\}=\emptyset, \forall j\neq k\bigg\}. 
 $$
 Let $E_{\lambda_0}=\{ x\in I: f_0(x)\ge\lambda_0\}$ and observe that $E_{\lambda_0}=\cup_{j\in J} E_{\lambda_0, j},$ where 
 $\text{card}\, J<n+1$, $E_{\lambda_0, j}$ are pairwise disjoint intervals such that    
 \begin{equation}\label{impor}
 E_{\lambda_0, j}\cap S= \cup^*{S_{j}}, 
 \end{equation}
  $\lambda_0|E_{\lambda_0}|=|S|$ and, in fact, for every $j$, 
 \begin{equation}\label{egs}
  \lambda_0|E_{\lambda_0, j}|=|E_{\lambda_0, j}\cap S|.
  \end{equation}
 
 Now,  by induction hypothesis, there exists a positive function $g$ supported in $I$  such that:
 
\noindent
(i') $g(x)=1$, for every $x\in E_{\lambda_0}$. 

\noindent
(ii') $g(x)\ge\frac{|E_{\lambda_0}|}{|I|}$, for every $x\in I$.

\noindent
(iii') For every $\frac{|E_{\lambda_0}|}{|I|}<\lambda\le 1$, the level set
$$
\{x: g(x)\ge \lambda\}=\cup_k J_{k, \lambda}', 
$$
satisfying  that $\{J_{k, \lambda}'\}_k$ are pairwise disjoint intervals, 
\begin{equation}\label{nuevo}
|E_{\lambda_0}\cap J_{k, \lambda}'|=\lambda |J_{k, \lambda}'|, \qquad E_{\lambda_0}\cap J_{k, \lambda}'=\cup^* E_{\lambda_0, j}, 
\end{equation}
and 
\begin{equation}\label{hh}
\{x: g(x)\ge \lambda\}\cap E_{\lambda_0}= \cup^* {E_{\lambda_0, j}}. 
\end{equation}

Then,  we claim that the function $f_{S,I}$ defined by
$$
f_{S,I}(x)= f_0(x), \quad \text{if }x\in E_{\lambda_0},  \qquad f_{S,I}(x)= \lambda_0 g(x), \quad  \text{if }x\in I\setminus E_{\lambda_0},
$$
satisfies all the required conditions (see Figure~\ref{fig2}). Clearly (i) and (ii) hold true. To see (iii) we divide it in two cases:

\noindent
Case 1.- If $\lambda_0\le\lambda \le 1$, 
$$
\{x: f_{S,I}(x)\ge \lambda\}=    \{x\in E_{\lambda_0}: f_0(x)\ge\lambda\}=   \bigcup_{k=1}^{n+1} \{x\in I: f_k(x)\ge\lambda\} = \bigcup_k J_{k, \lambda}, 
$$
and the result follows easily. 

\noindent
Case 2.- $\frac{|S|}{|I|}<\lambda<\lambda_0$. In this case
$$
\{x: f_{S,I}(x)\ge\lambda\}= \{x: g(x)\ge\lambda/\lambda_0\}, 
$$
and since $\frac{|E_{\lambda_0}|}{|I|}= \frac{|S|}{\lambda_0|I|}<\frac \lambda{\lambda_0}<1$, we can apply (iii') and the properties of $E_{\lambda_0}$ to conclude that 
$$
\{x: f_{S,I}(x)\ge\lambda\}=\bigcup_k J_{k, \lambda/\lambda_0}', 
$$
with  $\{ J_{k, \lambda/\lambda_0}'\}_k$   pairwise disjoint intervals satisfying
$$
|E_{\lambda_0}\cap J_{k, \lambda/\lambda_0}'|=\frac \lambda{\lambda_0} |J_{k, \lambda/{\lambda_0}}'|.
$$
So we have  to prove that 
\begin{figure}[b] 
   \centering
   \includegraphics[width=\textwidth]{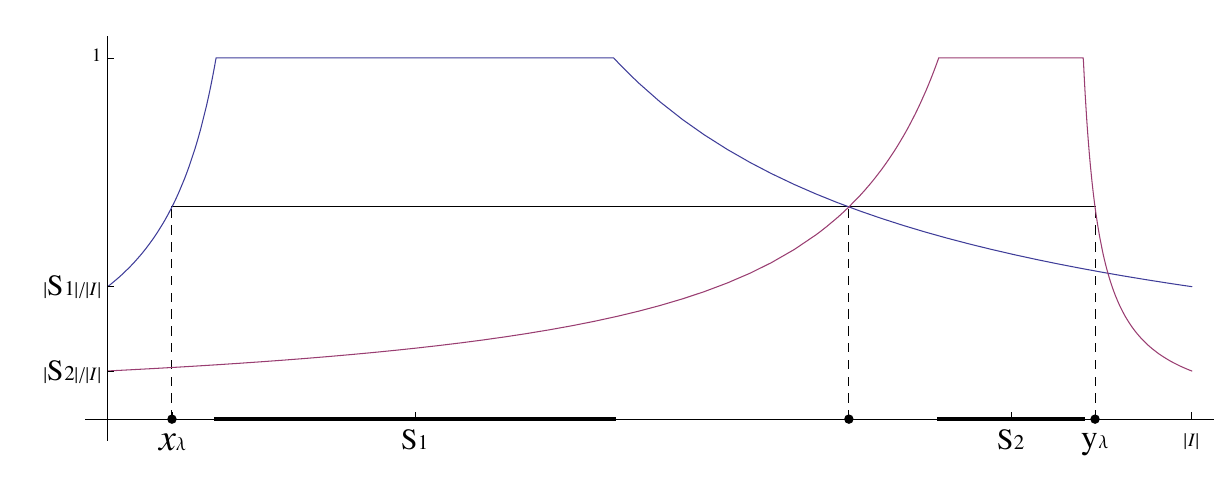} 
   \caption{$f_{S,I}$ when $m=2$.}
   \label{fig2}
\end{figure}
$$
|S\cap J_{k, \lambda/\lambda_0}'|=\lambda_0 |  E_{\lambda_0}\cap J_{k, \lambda/\lambda_0}'|.
$$

Now, from (\ref{egs}) and  (\ref{nuevo}), we obtain that
$$
\lambda_0 |  E_{\lambda_0}\cap J_{k, \lambda/\lambda_0}'|=\lambda_0 \sum^* |E_{\lambda_0, j}|=|\cup^* E_{\lambda_0, j}\cap S|=|S\cap J_{k, \lambda/\lambda_0}'|.
$$
Finally, using (\ref{impor}) and (\ref{hh}), we obtain that 
$$
\{x: g(x)>\lambda\}\cap S=  \{x: g(x)>\lambda\}\cap E_{\lambda_0}\cap S=     \cup^* {E_{\lambda_0, j}}\cap S= \cup^*
S_i,
$$
and the result follows. 
\end{proof}

\begin{lemma}\label{mean}
Let $S$ be a subset of the interval $I$  such that it is a union of  pairwise disjoint intervals: $S=\cup_{k=1}^{N}S_k$. If  $s=\frac{|I|}{|S|}$, then
$$
\frac{1}{|I|}\int_I {f_{S, I}(x)}dx=\frac{1+\log s}{s}.
$$
\end{lemma}
\begin{proof}
We observe that by construction of the function $f_{S, I}$ we have that 
$$
|\{x:f_{S, I}(x)\ge \la\}|=\begin{cases}
|I|,&\mbox{if $\la\in(0, 1/s)$}\\
|S|/\la,&\mbox{if $\la\in [1/s,1]$}\\
0,&\mbox{if $\la>1$ }.
\end{cases}
$$
Then
\begin{align*}
\frac{1}{|I|}\int_If_{S, I}(x) dx&=\frac{1}{|I|}\int_0^\infty |\{x:f_{S, I}(x)\ge \la\}|d\la=\frac{1+\log s}{s}.
\end{align*}

\end{proof}

\begin{proof}[Proof of Theorem \ref{BpiBp}]
Let $(I_j)_{j=1}^J$ be a finite family of pairwise disjoint intervals and let $(S_j)_{j=1}^{J}$ be 
such that  $S_j\subseteq I_j$, $S_j$  a finite union of pairwise disjoint intervals
with   ${|I_j|}/{|S_j|}=s$,  for every $j$.  Let \begin{equation}
f(x)=\sum_{j=1}^J f_{S_j, I_j} (x).
\end{equation}

 By the weak-type boundedness of $M$ we get, for every $t>0$, 
\begin{equation}\label{part 1}
W(u(\{x\in\R:\ Mf(x)>t\}))\lesssim \frac{1}{t^p}||f||_{\Lambda ^p_u(w)}^p.
\end{equation}
Now, 
\begin{align*}
||f||_{\Lambda ^p_u(w)}^p&=\int_0^\infty p\lambda^{p-1} W(u(\{x:\ f(x)>\lambda\}))d\lambda\\
&\leq\int_0^{1/s} p\lambda^{p-1} W(u(\{x:\ f(x)>\lambda\}))d\lambda\\
&\qquad+\int_{1/s}^1 p\lambda^{p-1} W(u(\{x:\ f(x)>\lambda\}))d\lambda
=I+II. 
\end{align*}
By Remark \ref{rem} (ii)  we have that 
$$
\frac{W\left(u\left(\bigcup_{j=1}^J I_j\right)\right)}{W\left(u\left(\bigcup_{j=1}^J S_j\right)\right)}
       \lesssim \max_{1\leq j\leq J} \left(\frac{|I_j|}{|S_j|}\right)^p\approx s^p, 
$$
and so
\begin{align*}
I\lesssim \int_0^{1/s} \la^{p-1} s^p W(u\left(\cup_{j=1}^J S_j\right))d\la\approx W(u\left(\cup_{j=1}^J S_j\right)).
\end{align*}

On the other hand, by Lemma~\ref{mod}, if $\la\in(1/s,1)$, the set $J_\la=\{x:\ f(x)>\la\}$ 
is the union of disjoint intervals $J_{\la,k}$ such that, for every $k$, 
$$
\frac{|J_{\la,k}|}{|S\cap J_{\la,k}|}=\frac{1}{\la}, $$
and $S=\cup_{j=1}^J S_j\subseteq J_\la$.
Therefore, 
$$
\frac{W\left(u\left(\bigcup_{k}J_{\la,k}\right)\right)}{W\left(u\left(S\right)\right)}
=\frac{W\left(u\left(\bigcup_{k}J_{\la,k}\right)\right)}{W\left(u\left(\bigcup_{k} S\cap J_{\la,k}\right)\right)} 
\lesssim \max_{k} \bigg(\frac{|J_{\la,k}|}{|S\cap J_{\la,k}|}\bigg)\approx \la^{-p}.
$$
Hence
\begin{align*}
II&\lesssim \int_{1/s}^1 \la^{p-1} \la^{-p} W(u\left(S\right))d\la \approx (1+\log s)W(u\left(S\right))\\
&\approx(1+\log s)W(u\left(\cup_{j=1}^J S_j\right)).
\end{align*}
So, we have that
\begin{equation}\label{part 2}
||f||_{\Lambda^p_u(w)}^p\lesssim (1+\log s)W(u\left(\cup_{j=1}^J S_j\right)).
\end{equation}

On the other hand, for every $j$, 
$$
I_j\subseteq \bigg\{x\in\R:\ Mf(x)>\frac{1}{2|I_j|}\int_{I_j} f(x)dx \bigg\},
$$
and, by Lemma~\ref{mean},  for every $j$, 
$$
\frac{1}{|I_j|}\int_{I_j} f(x)dx=\frac{1+\log s}{s}.
$$
Hence,
\begin{equation}\label{part 3}
W(u(\cup_j I_j))\leq W\left(u\left(\left\{x\in\R:\ Mf(x)>(1+\log s)/2s\right\}\right)\right).
\end{equation}
Finally,  if we fix $\displaystyle t=(1+\log s)/2s$ in \eqref{part 1}, and combine \eqref{part 2} and  \eqref{part 3} 
we obtain
$$
\frac{W(u(\cup_j I_j))}{W(u\left(\cup_j S_j\right))}\lesssim  (1+\log s)^{1-p} s^p .
$$
Then, taking supremum,  we obtain that 
$$
\overline{W}_{u}(s)\lesssim  (1+\log s)^{1-p} s^p
$$
and by Lemma~\ref{lll}, it follows that $\overline{W}_{u}(t)\lesssim t^{q}$,  for some $q<p$, as we wanted to see. 
\end{proof}

Theorem~\ref{BpiBp} not only concludes the proof of the Boyd Theorem in weighted Lorentz spaces. 
As we mentioned at the end of the Section 4, it also implies the following result:

\begin{theorem}\label{him}  Let   $p>1$ . If 
$$
H:\Lambda^{p}_u(w)\longrightarrow \Lambda^{p}_u(w)
$$
is bounded, then
$$
M:\Lambda^{p}_u(w)\longrightarrow \Lambda^{p}_u(w)
$$
is also bounded.
\end{theorem}

Finally, as in the proof of the  characterization of the weak-type boundedness given in \cite{acs2:acs2}, 
 and using Theorem~\ref{BpiBp}, we  can also characterize the boundedness of $H$ on $\Lambda^p_u(w)$, for $p>1$, in terms of geometric conditions on the weights $u$ and $w$ as follows:
 
\begin{theorem} If $p> 1$,  then
$$
H:\Lambda^{p}_u(w)\longrightarrow \Lambda^{p}_u(w)
$$
is bounded  if and only if the  following three conditions hold: 

\noindent
(i) $u\in A_\infty$.

\noindent
(ii) $w\in B^*_\infty$.

\noindent
(iii) Condition (\ref{raposo}) holds. 
\end{theorem}

\end{document}